\documentclass[12pt]{amsart}
\usepackage{amsmath,amssymb, amsfonts,amscd,
psfrag,graphicx}

\newcommand{\lebn}

\theoremstyle{plain}
\newtheorem{prop}[equation]{Proposition}

\newtheorem{thm}[equation]{Theorem}

\newtheorem{cor}[equation]{Corollary}
\newtheorem{lem}[equation]{Lemma}

\theoremstyle{definition}

\newtheorem{defn}[equation]{Definition}

\numberwithin{equation}{section}

\DeclareMathOperator{\NP}{NP}

\newcommand{\B}{\mathcal{B}}

\newcommand{\RE}{\mathcal{R}}

\newcommand{\IE}{\mathcal{I}}
\newcommand{\HE}{\mathcal{H}}
\newcommand{\XE}{\mathcal{X}}
\newcommand{\YE}{\mathcal{Y}}
\newcommand{\ZE}{\mathcal{Z}}
\newcommand{\LE}{\mathcal{L}}

\newcommand{\kk}{\Bbbk}
\newcommand{\lk}{\textrm{lk}}
\newcommand{\del}{\textrm{del}}
\newcommand{\Del}{\textrm{Del}}
\newcommand{\Add}{\textrm{Add}}

\newcommand{\cd}{\textrm{cochord}}
\newcommand{\reg}{\textrm{reg}}

\newcommand{\im}{\textrm{im}}

\newcommand{\D}{\Delta}

\begin{document}

\bibliographystyle{plain}

\title[Regularity and Lozin's transformation]{Bounding Castelnuovo-Mumford regularity of graphs via Lozin's transformation}
\author{T\" urker B\i y\i ko$\breve{g}$lu and Yusuf Civan}

\address{Department of Mathematics, Izmir Institute of Technology,
Gulbahce, 35437, Urla, Izmir, Turkey}

\address{Department of Mathematics, Suleyman Demirel University,
Isparta, 32260, Turkey.}

\email{tbiyikoglu@gmail.com\\
yusufcivan@sdu.edu.tr}


\date{\today}

\thanks{Both authors are supported by T\" UBA through Young Scientist Award
Program (T\" UBA-GEB\. IP/2009/06 and 2008/08) and by T\" UB\. ITAK, grant no:111T704}


\begin{abstract}
We prove that when a Lozin's transformation is applied to a graph, the (Castelnuovo-Mumford) regularity of the graph increases
exactly by one, as it happens to its induced matching number. As a consequence, we show that the regularity of a graph can be bounded 
from above by a function of its induced matching number. We also prove that the regularity of a graph is always less than or equal to the sum
of its induced matching and decycling numbers.
\end{abstract} 

\maketitle
%%%%%%%%%%%%%%%%%%%%%%%%%%%%%%%%%%%%%%%%%%%%%%%%%%%%%%%%%%%%
\section{Introduction}

Castelnuovo-Mumford regularity (or just the regularity) is something of a two-way study in the sense that it is a fundamental invariant
both in commutative algebra and discrete geometry. The regularity is a kind of universal bound for measuring the complexity of a considered 
object (a module, a sheaf or a simplicial complex). Our interest here is to compute or provide better bounds for the regularity of
rings associated to graphs. We are primarily concerned on the computation of the regularity of the edge ring 
(or equivalently the Stanley-Reisner ring of the independence complex) of a given graph. One way to attack such a problem goes by translating
the underlying algebraic or topological language to that of graph's. Such an approach may enable us to bound the regularity of a graph via
other graph parameters, and the most likely candidate is the induced matching number. By a theorem of Katzman~\cite{MK},
it is already known that the induced matching number provides a lower bound for the regularity of a graph, and the characterization of graphs in which
the regularity equals to the induced matching number has been the subject of many recent papers~\cite{BC,HVT,MMCRTY,MV,VT,RW2}. 

During his search on the complexity of the induced matching number, Lozin~\cite{VVL} describes an operation 
(he calls it as the \emph{stretching operation}) on graphs, and he proves that when it is applied to a graph, 
the induced matching number increases exactly by one. His operation works simply by considering a vertex $x$ of a graph $G$
whose (open) neighborhood splitted into two disjoint parts $N_G(x)=Y_1\cup Y_2$, and replacing the vertex $x$ with a 
four-path on $\{y_1,a,b,y_2\}$ together with edges $uy_i$ for any $u\in Y_i$ and $i=1,2$ (see Section~\ref{section:reg-lozin}).
One of the interesting results
of his work is that the induced matching problem remains $\NP$-hard in a narrow subclass of bipartite graphs.
We here prove that his operation has a similar effect on the regularity:

\begin{thm}\label{thm:lozin+reg}
Let $G=(V,E)$ be a graph and let $x\in V$ be given. Then $\reg(\LE_x(G))=\reg(G)+1$, where $\LE_x(G)$ is
the Lozin's transform of $G$ with respect to the vertex $x$.
\end{thm}

Theorem~\ref{thm:lozin+reg} implies that the computational complexity of the regularity of arbitrary graphs is equivalent to
that of bipartite graphs having sufficiently large girth with maximum degree three.

One of the main advantages of Theorem~\ref{thm:lozin+reg} is that we are able to prove that the regularity of any graph can be bounded
above by a function of its induced matching number. Moreover, we also show that the regularity of a graph is always less than or equal to
the sum of its induced matching and decycling numbers.

We further employ the language of graph coloring theory by introducing
a Gallai type graph~\cite{BLS, KC} associated to any given graph in order to describe a new class of graphs in which the regularity equals to the
induced matching number. 

The organization of the paper is as follows. In Section $2$, we fix our notations needed throughout, recall definitions and basic properties
of the regularity of simplicial complexes and graphs. The Section $3$ is devoted to the proof
of Theorem~\ref{thm:lozin+reg}, where we also analyze the topological effect of Lozin's transformation on the independence complexes of graphs.
In the final section, we provide various bounds on the regularity as promised.

%%%%%%%%%%%%%%%%%%%%%%%%%%%%%%%%%%%%%%%%%%%%%%%%%%%%%%%%%%%%%%%%%%%%%%%%%%%%%%%%%%%%%%%%%%%%%%%%%%
\section{Preliminaries}

Let $\D$ be a simplicial complex on the vertex set $V$, and let $\kk$ be any field. Then 
the \emph{Castelnuovo-Mumford regularity} (or just the \emph{regularity}) $\reg_{\kk}(\D)$ of $\D$ over $\kk$
is defined by 
\begin{equation*}
\reg_{\kk}(\D):=\max \{j\colon \widetilde{H}_{j-1}(\D[S];\kk)\neq 0\;\textrm{for\;some}\;S\subseteq V\},
\end{equation*}
where $\D[S]:=\{F\in \D\colon F\subseteq S\}$ is the induced subcomplex of $\D$ by $S$, and $\widetilde{H}_{*}(-;\kk)$ denotes
the (reduced) singular homology. Note that this definition of the regularity coincides with the algebraic one 
via the well-known Hochster's formula. 

Some of the immediate consequences of the above definition are as follows. Firstly, the regularity is dependent on the characteristic of
the coefficient field (compare Example $3.6$ of \cite{MV}). Secondly, it is not a topological invariant, and it is monotone decreasing with respect to
the induced subcomplex operation, that is, $\reg_{\kk}(\D[K])\leq \reg_{\kk}(\D)$ for any $K\subseteq V$. In most cases, our results are independent
of the choice of the coefficient field, so we drop $\kk$ from our notation.

Even if the regularity is not a topological invariant, the use of topological methods plays certain roles. In many cases, we will appeal
to an induction on the cardinality of the vertex set by a particular choice of a vertex accompanied by two subcomplexes. To be more explicit,
if $x$ is a vertex of $\D$, then the subcomplexes $\del_{\D}(x):=\{F\in \D\colon x\notin F\}$ and 
$\lk_{\D}(x):=\{R\in \D\colon x\notin R\;\textrm{and}\;R\cup \{x\}\in \D\}$ are called the \emph{deletion} and \emph{link} of $x$ in $\D$ respectively. 
Such an association brings the use of a Mayer-Vietoris sequence of the pair $(\D,x)$:
\begin{equation*}
\cdots \to \widetilde{H}_{j}(\lk_{\D}(x)) \to \widetilde{H}_{j}(\del_{\D}(x))\to \widetilde{H}_{j}(\D)\to \widetilde{H}_{j-1}(\lk_{\D}(x)) \to
\cdots \widetilde{H}_0(\D)\to 0.
\end{equation*}

\begin{prop}\label{prop:induction-sc}
Let $\D$ be a simplicial complex and let $x\in V$ be given. Then
\begin{equation*}
\reg (\D)\leq \max\{\reg (\del_{\D}(x)), \reg(\lk_{\D}(x))+1\}.
\end{equation*}
\end{prop}
\begin{proof}
Suppose that $\reg(\D)=k$, and let $W\subseteq V$ be a subset for which $\widetilde{H}_{k-1}(\D[W])\neq 0$. If $x\notin W$,
then $W\subseteq V(\del_{\D}(x))$ so that $\reg(\del_{\D}(x))\geq k$, that is, $\reg(\del_{\D}(x))=k$. Therefore, we may assume that
$x\in W$. 

We set $\D_0:=\D[W]$, $\D_1:=\del_{\D}(x)[W]$ and $\D_2:=\lk_{\D}(x)[W]$, and consider the Mayer-Vietoris sequence of the pair
$(\D_0,x)$:
\begin{equation*}
\cdots \to \widetilde{H}_{j}(\D_2) \to \widetilde{H}_{j}(\D_1)\to \widetilde{H}_{j}(\D_0)\to \widetilde{H}_{j-1}(\D_2) \to
\cdots \widetilde{H}_0(\D_0;\kk)\to 0.
\end{equation*}
Observe that $\D_1=\del_{\D_0}(x)$ and $\D_2=\lk_{\D_0}(x)$.
Now, if $\reg(\del_{\D}(x))<k$, then $\widetilde{H}_{k-1}(\D_1)=0$ so that $\widetilde{H}_{k-2}(\D_2)\neq 0$, since
$\widetilde{H}_{k-1}(\D_0)\neq 0$ by our assumption.
Thus, $\reg(\lk_{\D}(x))\geq k-1$. This proves the claim.
\end{proof}

We next review some necessary terminology from graph theory. 
By a graph $G=(V,E)$, we will mean an undirected graph without loops or
multiple edges. An edge between $u$ and $v$ is denoted by $e=uv$ or
$e=(u,v)$ interchangeably. A graph $G=(V,E)$ is called an \emph{edgeless graph} on $V$ whenever
$E=\emptyset$. If $U\subset V$, the graph induced on $U$ is written $G[U]$, and in particular,
we abbreviate $G[V\backslash U]$ to $G-U$, and write $G-x$ whenever $U=\{x\}$. In a similar vein,
we denote by $G-F$ the graph on $V$ having the edge set $E\backslash F$ whenever $F\subseteq E$.

The {\it complement} $\overline{G}:=(V,\overline{E})$ of a graph $G=(V,E)$ is the graph
on $V$ with $uv\in \overline{E}$ if and only if $uv\notin E$.
Throughout $K_n$, $C_n$ and $P_n$ will denote the complete, cycle and path graphs on $n$ vertices
respectively. Given two nonempty graphs  $G_1=(V_1,E_1)$ and $G_2=(V_2,E_2)$ on disjoint sets,
then their {\it disjoint union} $G_1\cup G_2$ is defined to be the graph $G_1\cup G_2=(V_1\cup V_2,E_1\cup E_2)$.

For a given subset $U\subseteq V$, the (open) neighborhood of $U$ is
defined by $N_G(U):=\cup_{u\in U}N_G(u)$, where $N_G(u):=\{v\in V\colon uv\in E\}$,
and similarly, $N_G[U]:=N_G(U)\cup U$ is the 
closed neighborhood of $U$. Furthermore,
if $F=\{e_1,\ldots,e_k\}$ is a subset of edges of $G$, we write $N_G[F]$ for the set
$N_G[V(F)]$, where $V(F)$ is the set of vertices incident to edges in $F$.

A subset $I\subseteq V$ is called an {\it independent set} whenever $G[I]$ is an edgeless graph.
The set of all independent sets of $G$ forms a simplicial complex $\IE(G)$, the {\it independence complex}
of $G$. The largest cardinality of an independent set in $G$ is called the {\it independence number} of $G$ and
denoted by $\alpha(G)$. 

\begin{defn}
Let $G$ be a graph. The {\it regularity} of $G$ is defined by $\reg(G):=\reg(\IE(G))$.
\end{defn}

When considering the complex $\IE(G)$, the deletion and link of a given vertex $x$ correspond to the independence complexes
of induced subgraphs, namely that $\del_{\IE(G)}(x)=\IE(G-x)$ and $\lk_{\IE(G)}(x)=\IE(G-N_G[x])$. Therefore, the following is an immediate
consequence of Proposition~\ref{prop:induction-sc}, which is also proven in algebraic setting in~\cite{MV}.

\begin{cor}\label{cor:induction-sc}
Let $G$ be a graph and let $v\in V$ be given. Then
\begin{equation*}
\reg (G)\leq \max\{\reg (G-v), \reg(G-N_G[v])+1\}.
\end{equation*}
\end{cor}

We say that $G$ is $H$-free if no induced subgraph of $G$ is isomorphic to $H$.
A graph $G$ is called \emph{chordal} if it is $C_k$-free for any $k\geq 4$, and a graph is said to be
\emph{cochordal} if its complement is a chordal graph.
A subset $S\subseteq V$ is called a {\it complete} of $G$ if $G[S]$ is isomorphic to a complete
graph, and the largest cardinality of a complete of $G$ is called the \emph{clique number} of $G$ and
denoted by $\omega(G)$.

Recall that a subset $M\subseteq E$ is called a {\it matching} of $G$ if no two edges in $M$ share a common end. 
Moreover, a matching $M$ of $G$ is an {\it induced matching}
if it occurs as an induced subgraph of $G$, and the cardinality of a maximum induced matching is called the
{\it induced matching number} of $G$ and denoted by $\im(G)$. 

Let $G$ be a graph, and $\HE$ be a family of graphs. The $\HE$-{\it cover number} of $G$ is
the minimum number of subgraphs $H_1,\ldots,H_r$ of $G$ such that every $H_i\in \HE$ and $\cup E(H_i)=E(G)$.
In particular, we denote by $\cd(G)$, the {\it cochordal cover number} of $G$ (see~\cite{RW2}). 

\begin{thm}[\cite{MK, RW2}]\label{thm:im-reg-cd}
The inequality $\im(G)\leq \reg(G)\leq \cd(G)$ holds for any graph $G$.
\end{thm}
Both bounds are far from being tight even for connected graphs. We have constructed in~\cite{BC} a connected
graph $G_n$ for each $n\geq 1$ such that $\cd(G_n)=\reg(G_n)+n$ (see Proposition 3.12 of~\cite{BC}). For the lower bound,
let $R_n$ be the graph obtained from $(n+1)$ disjoint five cycles by adding an extra vertex and connecting it to exactly one vertex
of each five cycle (see Figure~\ref{fig-reg-im-1}). Then the equality $\reg(R_n)=\im(R_n)+n$ for any $n\geq 1$ follows from
Corollary~\ref{cor:induction-sc}. 
\begin{figure}[ht]
\begin{center}
\includegraphics[width=3.1in,height=1in]{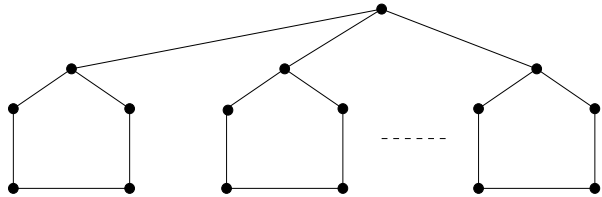}
\end{center}
\caption{The graph $R_n$.}
\label{fig-reg-im-1}
\end{figure}
The graph $R_n$ is also interesting for one other reason that we explain next.
A vertex $x$ of $G$ is called a {\it shedding vertex} if for every independent set $S$ in $G-N_G[x]$, there is some vertex $v\in N_G(x)$
so that $S\cup \{v\}$ is independent. A graph $G$ is called {\it vertex decomposable} if either it is an edgeless graph or it has a shedding
vertex $x$ such that $G-x$ and $G-N_G[x]$ are both vertex-decomposable.

\begin{thm}[\cite{BC}]\label{thm:vd-reg}
If $G$ is a $(C_4,C_5)$-free vertex decomposable graph, then $\reg(G)=\im(G)$.
\end{thm}

The bounds of Theorem~\ref{thm:im-reg-cd} can be tight if the graph lies in a specific graph class.

\begin{prop}[\cite{HVT,RW2}]\label{prop:chordal-reg}
If $G$ is a chordal graph, then $\im(G)=\reg(G)=\cd(G)$.
\end{prop}

The existence of vertices satisfying some extra properties is useful when dealing with the homotopy type of the independence complexes
of graphs.
\begin{thm}[\cite{AE, MT}]\label{thm:hom-induction}
If $N_G(u)\subseteq N_G(v)$, then there is a homotopy equivalence $\IE(G)\simeq \IE(G-v)$. On the other hand,
if $N_G[u]\subseteq N_G[v]$, then the homotopy equivalence $\IE(G)\simeq \IE(G-v)\vee \Sigma \IE(G-N_G[v])$
holds, where $\Sigma X$ denotes the (unreduced) suspension of $X$.
\end{thm}

\begin{defn}
An edge $e=(u,v)$ is called an {\it isolating edge} of $G$ with respect to a vertex $w$, if $w$ is an isolated vertex of
$G-N_G[e]$.
\end{defn}

\begin{thm}[\cite{MA}]\label{thm:isolating}
If $\IE(G-N_G[e])$ is contractible, then the natural inclusion $\IE(G)\hookrightarrow \IE(G-e)$ is a homotopy equivalence.
\end{thm}

In particular, Theorem~\ref{thm:isolating} implies that $\IE(G)\simeq \IE(G-e)$ whenever the edge $e$ is isolating. 
This brings the use of an operation, adding or removing an edge, on a graph without altering its homotopy type. We will
follow \cite{MA} to write $\Add(x,y;w)$ (respectively $\Del(x,y;w)$) to indicate that we add the edge
$e=xy$ to (resp. remove the edge $e=xy$ from) the graph $G$, where $w$ is the corresponding isolated vertex.

%%%%%%%%%%%%%%%%%%%%%%%%%%%%%%%%%%%%%%%%%%%%%%%%%%%%%%%%%%%%%%%%%%%%%%%%%%%%%%%%%%%%%%%%%%%
\section{Regularity of graphs and Lozin's transformation}\label{section:reg-lozin}

This section is devoted to the proof of Theorem~\ref{thm:lozin+reg}. We begin with recalling the definition of
the Lozin's transformation on graphs~\cite{VVL}.

Let $G=(V,E)$ be a graph and let $x\in V$ be given. The \emph{Lozin's transform} $\LE_x(G)$ of $G$ with respect to
the vertex $x$ is defined as follows:
\begin{itemize}
\item[(i)] partition the neighborhood $N_G(x)$ of the vertex $x$ into two subsets $Y$ and $Z$ in arbitrary way;\\
\item[(ii)] delete vertex $x$ from the graph together with incident edges;\\
\item[(iii)] add a $P_4=(\{y,a,b,z\},\{ya,ab,bz\})$ to the rest of the graph;\\
\item[(iv)] connect vertex $y$ of the $P_4$ to each vertex in $Y$, and connect $z$ to each vertex in $Z$.
\end{itemize}

\begin{figure}[ht]
\begin{center}
\psfrag{x}{$x$}\psfrag{y}{$y$}\psfrag{a}{$a$}\psfrag{b}{$b$}
\psfrag{z}{$z$}\psfrag{Y}{$Y$}\psfrag{Z}{$Z$}
\includegraphics[width=2.5in,height=0.8in]{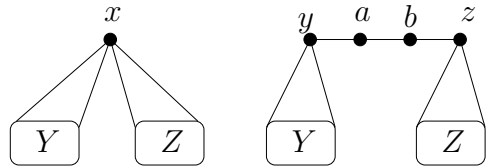}
\end{center}
\caption{The Lozin's transformation.}
\label{lozin-1}
\end{figure}

When the decomposition $N_G(x)=Y\cup Z$ is of importance, we will write $\LE_x(G;Y,Z)$ instead of $\LE_x(G)$.
It should be noted that we allow one of the sets $Y$ and $Z$ to be an empty set. Furthermore, if $x$ is an isolated vertex,
the corresponding Lozin's transform of $G$ with respect to the vertex $x$ is the graph $(G-x)\cup P_4$.
Throughout we will follow the convention made in~\cite{VVL} to use special notations for some particular classes of graphs:\\
$\XE_k$, the class of $(C_3,C_4,\ldots,C_k)$-free graphs,\\
$\YE_l$, the class of $(H_1,H_2,\ldots,H_l)$-free graphs,\\
$\ZE_3$, the class of graphs with maximum degree $3$,\\
$\B$, the class of bipartite graphs,\\
where the graph $H_n$ is depicted in Figure~\ref{fig-H}.

\begin{figure}[ht]
\begin{center}
\psfrag{1}{$1$}\psfrag{2}{$2$}\psfrag{n}{$n$}
\includegraphics[width=2.5in,height=0.8in]{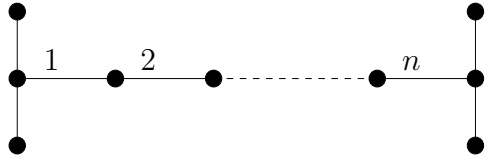}
\end{center}
\caption{Graph $H_n$.}
\label{fig-H}
\end{figure}

Lozin proves in \cite{VVL} that his transformation increases the induced matching number exactly by one. 
\begin{lem}[\cite{VVL}]\label{lem:lozin-inmatch}
For any graph $G$ and any vertex $x\in V(G)$, the equality $im(\LE_x(G;Y,Z))=im(G)+1$ holds.
Furthermore, any graph can be transformed by a sequence of Lozin transformations
into a graph in class $\XE_k\cap \YE_l\cap \ZE_3\cap \B$ with any integer $k\geq 3$ and $l\geq 1$.
\end{lem}

We next prove that the Lozin's transformation has a similar effect on the regularity. In order to simplify the notation, 
we note that when we mention the homology or homotopy of a graph, we mean that of its independence complex, so 
whenever it is appropriate, we drop $\IE(-)$ from our notation.
%%%%%%%%%%%%%%%%%%%%%%%%
\begin{prop}\label{prop:reg+isolating}
Let $e=(u,v)$ be an isolating edge of $H$ with respect to the vertex $w$. If $w$ has a neighbor $x$ of degree two,
then $\reg(H)=\reg(H-e)$.
\end{prop}
\begin{proof}
Without loss of generality, we may assume that $N_H(x)=\{u,w\}$. We let $G:=H-e$, and assume that $\reg(H)=m$.
Let $S\subseteq V=V(H)=V(G)$ be a subset satisfying $\widetilde{H}_{m-1}(H[S])\neq 0$. We suppose that
$S$ is maximal in the sense that $\widetilde{H}_{m-1}(H[R])=0$ for any $S\subsetneq R$.

If $u\notin S$ or $v\notin S$, we clearly have $H[S]\cong G[S]$.
Thus, we may assume that $u,v\in S$. In this case, if $w\in S$, then $H[S]\simeq H[S]-e\cong G[S]$, since $e$ is then an isolating edge
of $H[S]$. So, suppose that $w\notin S$. If $x\in S$,
we conclude that $H[S]\simeq H[S]-v\cong G[S\backslash \{v\}]$, where the homotopy equivalence is due to the inclusion
$N_{H[S]}(x)\subseteq N_{H[S]}(v)$ (see Theorem~\ref{thm:hom-induction}). It follows that we may further assume $x\notin S$.
We then define $S^*:=S\cup \{x\}$, and consider the Mayer-Vietoris sequence of the pair $(H[S^*],x)$: 
\begin{equation*}
\cdots \to \widetilde{H}_{m-1}(H[S^*]-N_{H[S^*]}[x])\to \widetilde{H}_{m-1}(H[S^*]-x)\to \widetilde{H}_{m-1}(H[S^*]) \to \cdots. 
\end{equation*} 
Note that $\widetilde{H}_{m-1}(H[S^*])=0$ by the maximality of $S$. Moreover, since $H[S^*]-x\cong H[S]$,
we need to have $\widetilde{H}_{m-1}(H[S^*]-x)\neq 0$. Therefore, we conclude that $\widetilde{H}_{m-1}(H[S^*]-N_{H[S^*]}[x])\neq 0$.
But, $H[S^*]-N_{H[S^*]}[x]\cong H[S\backslash \{u\}]$ and $H[S\backslash \{u\}]\cong G[S\backslash \{u\}]$ so that we have
$\widetilde{H}_{m-1}(G[S\backslash \{u\}])\neq 0$. As a result, we deduce that $\reg(G)\geq \reg(H)$.  

Assume now that $\reg(G)=n$, and let $W\subseteq V$ be a maximal subset satisfying $\widetilde{H}_{n-1}(G[W])\neq 0$.
Similar to the previous cases, we may assume that $u,v\in W$ and $w\notin W$. Suppose first that
$x\in W$. Then, $G[W]\simeq \Sigma(G[W]-N_{G[W]}[u])$, since $x$ has degree one in $G[W]$ and $u$ is
its only neighbor. It follows that $\widetilde{H}_{n-2}(G[W]-N_{G[W]}[u])\neq 0$.
We define $K:=(W\backslash N_H(u))\cup \{v,x,w\}$. Since $e$ is an isolating edge
in $H[K]$, we have $H[K]\simeq H[K]-e$. However, the vertex $u$ is of degree one in $H[K]-e$ and $x$ is its only neighbor;
hence, $H[K]-e\simeq \Sigma((H[K]-e)-N_{H[K]-e}[x]))$ by Theorem~\ref{thm:hom-induction}. 
On the other hand, we have  $(H[K]-e)-N_{H[K]-e}[x]\cong G[W]-N_{G[W]}[u]$
so that $\widetilde{H}_{n-1}(H[K])\neq 0$. 

Finally, suppose now that $x\notin W$. As in the previous case, we let $W^*:=W\cup \{x\}$,
and consider the Mayer-Vietoris sequence of the pair $(G[W^*],x)$:
\begin{equation*}
\cdots \to \widetilde{H}_{n-1}(G[W^*]-N_{G[W^*]}[x])\to \widetilde{H}_{n-1}(G[W^*]-x)\to \widetilde{H}_{n-1}(G[W^*]) \to \cdots. 
\end{equation*} 
Again, we have $\widetilde{H}_{n-1}(G[W^*])=0$ by the maximality of $W$. Since $G[W^*]-x\cong G[W]$,
the group $\widetilde{H}_{n-1}(G[W^*]-x)$ is nontrivial that implies $\widetilde{H}_{n-1}(G[W^*]-N_{G[W^*]}[x])\neq 0$.
However, $G[W^*]-N_{G[W^*]}[x]\cong G[W\backslash \{u\}]$ and $G[W\backslash \{u\}]\cong H[W\backslash \{u\}]$ so that we have
$\widetilde{H}_{n-1}(H[W\backslash \{u\}])\neq 0$. As a result, we deduce that $\reg(H)\geq \reg(G)$. This completes
the proof.     
\end{proof}
%%%%%%%%%%%%%%%%%%%%%%%%%%%%
We next verify that the homotopy type of Lozin's transform 
can be deduced from the source graph.

\begin{lem}\label{lem:lozin+stable}
Let $G=(V,E)$ be a graph and let $x\in V$ be given. Then $\IE(\LE_x(G;Y,Z))\simeq \IE(\LE_x(G;Y',Z'))$
for any two distinct decompositions $\{Y,Z\}$ and $\{Y',Z'\}$ of $N_G(x)$. 
\end{lem}
\begin{proof}
To prove the claim, it is enough to verify that for a given decomposition $N_G(x)=Y\cup Z$ and a vertex $u\in Z$,
the complexes $\IE(\LE_x(G;Y,Z))$ and $\IE(\LE_x(G;Y\cup \{u\},Z\backslash \{u\})$ have the same homotopy type.
However, moving the vertex $u$ from $Z$ to $Y$
corresponds to the sequence of isolating operations $\Add(u,y;b)$ and $\Del(u,z;a)$ in $\LE_x(G;Y,Z)$;
therefore, the claim follows from Theorem~\ref{thm:isolating}.
\end{proof}

\begin{prop}\label{prop:lozin+stable}
Let $G=(V,E)$ be a graph and let $x\in V$ be given. Then $\IE(\LE_x(G))\simeq \Sigma(\IE(G))$.
\end{prop}
\begin{proof}
In view of Lemma~\ref{lem:lozin+stable}, it is sufficient to show that $\IE(\LE_x(G;N_G(x),\emptyset))\simeq \Sigma(\IE(G))$.
In such a case, we set $\RE_x(G)=\LE_x(G;N_G(x),\emptyset)$ and note that $N_{\RE_x(G)}(z)\subseteq N_{\RE_x(G)}(a)$ so that
$\IE(\RE_x(G))$ is homotopy equivalent to $\IE(\RE_x(G)-a)$, while the latter graph is clearly isomorphic to 
$G\cup K_2$, where $K_2$ is induced by the edge $(b,z)$. It then follows that $\IE(\RE_x(G))\simeq \Sigma(\IE(G))$
as required. 
\end{proof}

\begin{lem}\label{lem:lozin+reg}
Let $G=(V,E)$ be a graph and let $x\in V$ be given. Then $\reg(\LE_x(G;Y,Z))=\reg(\LE_x(G;Y',Z'))$
for any two distinct decompositions $\{Y,Z\}$ and $\{Y',Z'\}$ of $N_G(x)$. 
\end{lem}
\begin{proof}
We proceed as in the proof of Lemma~\ref{lem:lozin+stable}. So, let a decomposition $N_G(x)=Y\cup Z$ and a vertex $u\in Z$ be given.
We claim that $\reg(\LE_x(G;Y,Z))=\reg(\LE_x(G;Y\cup \{u\},Z\backslash \{u\})$. If we set $H:=\LE_x(G;Y,Z)\cup (u,y)$, then
the edge $(u,y)$ is an isolating edge of $H$ with respect to the vertex $b$, while $b$ has a neighbor $a$ of degree two.
Therefore, Proposition~\ref{prop:reg+isolating} applies, that is, $\reg(H)=\reg(\LE_x(G;Y,Z))$.
On the other hand, $(u,z)$ is also an isolating edge in $H$ with respect to $a$, and the vertex $a$ has a degree two neighbor, namely
$b$. Thus we have $\reg(H)=\reg(H-(u,z))=\reg(\LE_x(G;Y\cup \{u\},Z\backslash \{u\})$. So the claim follows.
\end{proof}

%%%%%%%%%%%%%%%%%%???????
\begin{proof}[{\bf Proof of Theorem~\ref{thm:lozin+reg}}]
Given a subset $S\subseteq V$. Assume first that $x\notin S$. If we define $S':=S\cup \{a,b\}$,
it follows that $\LE_x(G)[S']\cong G[S]\cup K_2$, where $K_2$ is induced by the edge $(a,b)$.
Therefore, the complex $\IE(\LE_x(G)[S'])$ is homotopy equivalent to the suspension of 
$\IE(G[S])$.

Suppose now that $x\in S$. In such a case, we consider $S':=(S\backslash \{x\})\cup \{y,a,b,z\}$, and write 
$H:=G[S]$ and $H':=\LE_x(G)[S']$. Note that the Lozin's transform $\LE_x(H)$ of $H$ with respect to the partition $N_H(x)=(S\cap Y)\cup (S\cap Z)$
is exactly isomorphic to $H'$. Therefore, we have $\IE(H')\simeq \Sigma(\IE(H))$ by Proposition~\ref{prop:lozin+stable}, 
which implies that $\reg(\LE_x(G))\geq \reg(G)+1$.

For the converse, by Lemma~\ref{lem:lozin+reg}, it is sufficient to show that $\reg(\LE_x(G;N_G(x),\emptyset))\leq \reg(G)+1$.
So,  set $\LE_x(G)=\LE_x(G;N_G(x),\emptyset)$. Then we have
\begin{equation*}
\reg (\LE_x(G))\leq \max\{\reg (\LE_x(G)-a), \reg(\LE_x(G)-N_{\LE_x(G)}[a])+1\}.
\end{equation*}
by Corollary~\ref{cor:induction-sc}. However, the graph $\LE_x(G)-a$ is isomorphic to 
$G\cup (b,z)$ so that $\reg(\LE_x(G)-a)=\reg(G)+1$. Moreover, the graph
$\LE_x(G)-N_{\LE_x(G)}[a]$ is isomorphic to $(G-x)\cup \{z\}$ in which $z$ is an isolated vertex.
It means that $\reg(\LE_x(G)-N_{\LE_x(G)}[a])=\reg(G-x)\leq \reg(G)$. Therefore, we conclude that
$\reg(\LE_x(G))\leq \reg(G)+1$ as claimed.
\end{proof}

A special case of Lozin's transformation can be obtained by taking one of the sets in the partition of $N_G(x)$ 
as a singleton. In other words, for a given $u\in N_G(x)$, we consider the operation $\LE_x(G;\{u\},N_G(x)\backslash \{u\})$.
Observe that the resulting graph can be obtained from $G$ by replacing the edge $(u,x)$ in $G$ by a path $u-y-a-b-x$,
which is called the {\it triple subdivision} of $G$ with respect to the edge $(u,x)$. 
Following Theorem~\ref{thm:lozin+reg}, we may readily describe the effect of a triple subdivision on regularity.

\begin{cor}\label{cor:lozin+edgesubdiv}
Let $G$ be a graph and let $e=(u,v)$ be an edge. If $\LE(G;e)$ is the graph obtained from $G$ by the triple subdivision of $e$,
then $\reg(\LE(G;e))=\reg(G)+1$. 
\end{cor}

%%%%%%%%%%%%%%%%%%%%%%%%%%%%%%%%%%%%%%%%%%%%%%%%%%%%%%%%%%%%%%%%%%%%%%%%%%%%%%%%%%%%%%%%%%%%%%%%%%%%%%%%%%%

\section{Bounding the regularity of graphs}

In this section, we provide various upper bounds on the regularity of graphs as an application of Theorem~\ref{thm:lozin+reg}.
We first recall that a $k$-(vertex) coloring of a graph $G=(V,E)$ is a surjective mapping $\kappa\colon V\to\{1,2,\ldots,k\}$
such that if $uv\in E$, then $\kappa(u)\neq \kappa(v)$, and the chromatic number $\chi(G)$ of $G$ is the least integer $d$
for which $G$ admits a $d$-coloring. Moreover, a graph $G$ is called \emph{perfect} if $\chi(G[A])=\omega(G[A])$ for any
$A\subseteq V$.
\begin{defn}
For a given graph $G$ with at least one edge, we define a graph $G^*$, whose vertices are the edges of $G$, that is, $V(G^*)=E(G)$, and
if $e$ and $f$ are two edges in $G$, then $ef\in E(G^*)$ if and only if $G[V(\{e,f\})]\cong 2K_2$.
\end{defn}
We remark that the graph $G^*$ constructed here is exactly the complement of that introduced in~\cite{KC}.
Furthermore, the induced matching number of $G$ equals to the clique number of $G^*$, i.e., $\im(G)=\omega(G^*)$,
and under suitable restrictions, we can say more:

\begin{thm}\label{thm:chromatic-cochord}
If $G$ is a $(C_3,C_5)$-free graph, then $\chi(G^*)=\cd(G)$.
\end{thm}
\begin{proof}
Assume that $\chi(G^*)=k$, and let $\kappa\colon E(G)\to [k]$ be a proper vertex coloring
of $G^*$. We consider color classes $E_i=\kappa^{-1}(i)$ and denote by $G_i$, the subgraph of $G$ with
$E(G_i)=E_i$ for any $i\in [k]$. We then claim that each $G_i$ is a cochordal subgraph of $G$. Observe first that
the graph $G_i$ can not contain $\overline{C_r}$ for any $r\geq 6$, since $G$ is triangle-free. Moreover, if $G_i$
contains $\overline{C_5}=C_5$, then $G$ must contain a chord that in turn creates a triangle in $G$ which is not possible.
On the other hand, since the set of edges of $G_i$ corresponds to a color class in $G^*$, it is necessarily $\overline{C_4}$-free.
Therefore, the family $\{G_1,\ldots,G_k\}$ is a cochordal edge cover of $G$; hence, $\cd(G)\leq k$.

Suppose now that $\cd(G)=n$, and let $\{W_1,\ldots,W_n\}$ be a cochordal edge covering family of $G$. We then define
$\mu\colon E(G)\to [n]$ by $\mu(e):=\min \{j\in [n]\colon e\in E(W_j)\}$, and claim that $\mu$ is a proper coloring of $G^*$.
For this, if $ef\in E(G^*)$ for some $e,f\in E(G)$, there exists no $s\in [n]$ for which $e,f\in E(W_s)$, since $W_s$ is cochordal.
We therefore have $\mu(e)\neq \mu(f)$. This proves that $\chi(G^*)\leq n$.
\end{proof}

By Theorem~\ref{thm:chromatic-cochord}, we note that the inequality $\chi(G^*)\leq \cd(G)$ holds for any graph $G$ containing 
at least one edge. Furthermore, the followings are also immediate from Theorem~\ref{thm:chromatic-cochord}:
\begin{cor}\label{cor:col-bound}
If $G$ is a $(C_3,C_5)$-free graph, then $\omega(G^*)\leq reg(G)\leq \chi(G^*)$.
\end{cor}

\begin{cor}\label{cor:perfect}
If $G$ is a $(C_3,C_5)$-free graph such that $G^*$ is a perfect graph, then $\reg(G)=\im(G)$.
\end{cor}

Observe that if $G$ satisfies the conditions of Corollary~\ref{cor:perfect}, it must be $C_k$-free for any $k\geq 7$. It will be
interesting to find a full characterization of such graphs.

One of the important problem in graph coloring theory is to provide an upper bound on the chromatic number of a graph in terms of its 
clique number. Under certain conditions on the source graphs, we next show that such a bound is possible for their associated graphs.
For any given integer $m\geq 1$, we denote by $\LE(m)$, the class of graphs contained in $\XE_{3m+3}\cap \YE_{3m+3}\cap \ZE_3\cap \B$.

\begin{thm}\label{thm:bound-m}
If $G\in \LE(m)$, then $\chi(G^*)\leq \frac{m+1}{m}\omega(G^*)$.
\end{thm}
\begin{proof}
We proceed by an induction on the size of $G$. We first note that if $G$ is a forest, then $\im(G)=\cd(G)$ by 
Proposition~\ref{prop:chordal-reg}; hence,
\begin{equation*}
\omega(G^*)=\im(G)=\cd(G)=\chi(G^*)
\end{equation*}
by Theorem~\ref{thm:chromatic-cochord}. So, we may assume that $G$ contains at least one induced cycle,
say $C$, of order at least $3m+4$.

{\bf Claim 1.} There exists an induced path of order $3m+3$ in $G$ such that all of its inner vertices are of degree two in $G$.

\emph{Proof of Claim 1.} If the induced cycle $C$ contains at most one vertex of degree three, we can clearly construct such a path on $C$. 
On the other hand, if $x,y\in V(C)$ are two vertices having degree three in $G$, then all vertices (except possibly $x$ and $y$) on the shortest 
path in $C$ connecting $x$ and $y$ must have degree two, since $G$ is $H_{3m+3}$-free. Moreover, the length of such a path is at least $3m+2$ 
which proves the claim.

We let $P$ be a such path in $G$ and denote by $e_1,e_2,\ldots,e_{3m+2}$, the edges of $P$.

{\bf Claim 2.} If $Q$ is a maximum clique of $G^*$, then $|Q\cap E(P)|\geq m$.

\emph{Proof of Claim 2.} We simply note that the intersection $Q\cap E(P)$ is minimized exactly when $Q\cap N_{G^*}(e_1)\neq \emptyset$ and
$Q\cap N_{G^*}(e_{3m+2})\neq \emptyset$. However, it then follows that the intersection $Q\cap \{e_3,e_4,\ldots, e_{3m}\}$ is at least $m$.

As a result, we have $\omega(G^*-E(P))\leq \omega(G^*)-m$ by Claim 2, and also note that $\chi(P^*)=m+1$ 
by Proposition~\ref{prop:chordal-reg}. Therefore, we conclude that 

\begin{align*}
\chi(G^*)&\leq \chi(G^*-E(P))+m+1,\\
&\leq (\frac{m+1}{m})\omega(G^*-E(P))+m+1,\\
&\leq (\frac{m+1}{m})(\omega(G^*)-m)+m+1=(\frac{m+1}{m})\omega(G^*),
\end{align*}
where the second inequality follows from the induction.
\end{proof}

\begin{cor}\label{cor:bound-m}
If $G\in \LE(m)$, then $\reg(G)\leq \frac{m+1}{m}\im(G)$.
\end{cor}

\begin{defn}
For a given graph $G$ and an integer $m\geq 1$, we define its \emph{$\textrm{m}^{\textrm{th}}$-Lozin index} $l_m(G)$ by 
\begin{align*}
l_m(G):=\min\{k\geq 0\colon &\textrm{there\;exists\;a\;sequence}\;\LE_1,\LE_2,\ldots,\LE_k\;\textrm{of\;Lozin\;transformations}\\
&\textrm{such\;that}\;\LE_k(\LE_{k-1}(\ldots(\LE_1(G))\ldots )))\in \LE(m)\}.
\end{align*}
\end{defn}
We sometimes write $G_m$ for the graph obtained from $G$ by applying $k$ Lozin's transformations on $G$ such that $G_m\in \LE(m)$ where
$l_m(G)=k$. The following is an immediate consequence of Corollary~\ref{cor:bound-m} and Theorem~\ref{thm:lozin+reg}.
\begin{thm}\label{thm:bound-lozin-index}
For any graph $G$ and any integer $m\geq 1$, we have 
\begin{equation*}
\reg(G)\leq (\frac{m+1}{m})\im(G)+\frac{l_m(G)}{m}.
\end{equation*}
\end{thm}

In the rest of this section, as an another application of Theorem~\ref{thm:lozin+reg}, 
we show that the regularity of a graph is always less than or equal to the sum of its
induced matching and decycling numbers. To achieve this, we first verify that the whiskering of a graph with respect to a
decycling set results a vertex decomposable graph which will be of independent interest.

We recall that a vertex $x\in V$ is called a \emph{leaf} if it has only one neighbor in $G$, and a \emph{pendant edge}
in $G$ is an edge which is incident to a leaf. The \emph{whisker} $W_S(G)$ of a graph $G=(V,E)$ with respect to a given
subset $S=\{s_1,\ldots,s_k\}\subseteq V$ is the graph constructed from $G$ by $V(W_S(G)):=V\cup \{s'_1,\ldots,s'_k\}$
and $E(W_S(G))):=E\cup \{s_1s'_1,\ldots,s_ks'_k\}$. In particular, we abbreviate $W_V(G)$ by $W(G)$ (see~\cite{RV}). 

For a graph $G$ and $D\subseteq V(G)$, if $G-D$ is acyclic, i.e., contains no induced cycle, then $D$ is said to be a 
\emph{decycling set} of $G$. The size of a smallest decycling set of $G$ is called the \emph{decycling number} of $G$
and denoted by $\nabla(G)$ (see~\cite{BV}).

\begin{prop}\label{prop:edge-sub-div-decy}
$\nabla(G)=\nabla(\LE(G;e))$ for any graph $G$ and any edge $e=(u,v)$ of $G$.
\end{prop}
\begin{proof}
Since any decycling set for $G$ would remain to be a decycling set after any triple subdivision,
the inequality $\nabla(\LE(G;e)\leq \nabla(G)$ trivially holds.

Suppose $S\subseteq V(\LE(G;e))$ is a minimum decycling set of $\LE(G;e)$, and let $u-x-y-z-v$ be the resulting path
obtained after the triple subdivision of $e$. Note that $|S\cap \{x,y,z\}|\leq 1$, since $S$ is minimal.
If $S\cap \{x,y,z\}=\emptyset$, then $S$ is a decycling set for $G$. So, assume that $S\cap \{x,y,z\}\neq \emptyset$.
In such a case, the vertices $u$ and $v$ can not both belong to $S$, since otherwise the set $S\backslash \{x,y,z\}$
would still be a decycling set for $\LE(G;e)$. Therefore, we may assume $u\notin S$ without loss of generality.
However, it then follows that $(S\backslash \{x,y,z\})\cup \{u\}$ is a decycling set for $G$. Thus, the inequality
$\nabla(G)\leq \nabla(\LE(G;e)$ holds.
\end{proof}

\begin{prop}\label{prop:whisker-vd}
If $S$ is a decycling set of a graph $G$, then the graph $W_S(G)$ is vertex decomposable.
\end{prop}
\begin{proof}
Consider  $x\in S$, and let $x'$ be the leaf neighbor of $x$ in $W_S(G)$. 
We note that the vertex $x$ is clearly a shedding vertex of $W_S(G)$ and 
\begin{align*}
(W_S(G)-x)-x'=&W_{S\backslash \{x\}}(G-x)\quad\textrm{and}\quad \\
&(W_S(G)-N_{W_S(G)}[x])-\{u'\colon u\in N_G(x)\}=W_{S\backslash N_G[x]}(G-N_G[x]).
\end{align*}
Moreover, the sets $S\backslash \{x\}$ and $S\backslash N_G[x]$ are decycling sets for graphs $G-x$ and $G-N_G[x]$ respectively; 
hence, both graphs $W_{S\backslash \{x\}}(G-x)$ 
and $W_{S\backslash N_G[x]}(G-N_G[x])$ are vertex decomposable by the induction on $|S|$. 
It then follows that the graphs $W_S(G)-x$ and $W_S(G)-N_{W_S(G)}[x]$
are vertex decomposable, since the vertex $x'$ is an isolated vertex of $W_S(G)-x$, and similarly, $\{u'\colon u\in N_G(x)\}$
is a set of isolated vertices in $W_S(G)-N_{W_S(G)}[x]$.
\end{proof}

\begin{thm}\label{thm:reg-decycling}
For any graph $G$, we have $\reg(G)\leq \im(G)+\nabla(G)$.
\end{thm}
\begin{proof}
We let $G'$ be a graph obtained from $G$ by a necessary number, say $k\geq 0$, of triple subdivisions on $G$ so that it is
a $(C_4,C_5)$-free graph. We first recall that $\nabla(G)=\nabla(G')$ by Proposition~\ref{prop:edge-sub-div-decy}.

Assume that $S$ is a minimal decycling set for $G'$. By Proposition~\ref{prop:whisker-vd}, the graph
$W_S(G')$ is vertex decomposable. We therefore have $\reg(W_S(G'))=\im(W_S(G'))$ by Theorem~\ref{thm:vd-reg}, since
$W_S(G')$ is also $(C_4,C_5)$-free. Now, it follows that 
\begin{equation*}
\reg(G)+k=\reg(G')\leq \reg(W_S(G'))=\im(W_S(G'))\leq \im(G')+\nabla(G')=\im(G)+k+\nabla(G)
\end{equation*}
so that $\reg(G)\leq \im(G)+\nabla(G)$ as required.
\end{proof}

We remark that Theorem~\ref{thm:reg-decycling} is particularly useful when the graph has a small decycling number. Having this in mind,
we recall that a graph $G$ is called \emph{unicyclic} provided that it has exactly one cycle. 
\begin{cor}
If $G$ is a unicyclic graph, then $\im(G)\leq \reg(G)\leq \im(G)+1$. 
\end{cor}

%%%%%%%%%%%%%%%%%%%%%%%
The following is proved in~\cite{RW2} in the language of very well-covered graphs, we here provide an alternative
one relying on Theorems~\ref{thm:lozin+reg}, \ref{thm:vd-reg} and \ref{thm:reg-decycling}.

\begin{prop}\label{prop:reg-whisker}
For any graph $G$, we have $\reg(W(G))=\im(W(G))=\alpha(G)$.
\end{prop}
\begin{proof}
We follow the proof of Theorem~\ref{thm:reg-decycling} and write $G'$ for the graph obtained from $G$ by $k$-triple subdivisions on $G$
so that $G'$ is $(C_4,C_5)$-free. Now, the set $V=V(G)$ is a decycling set for $G'$ so that $W_V(G')$ is vertex decomposable by 
Proposition~\ref{prop:whisker-vd}. Therefore, we have $\reg(W_V(G'))=\im(W_V(G'))$. However, the graph $W_V(G')$ can be constructed
from $W(G)$ by $k$-triple subdivisions on the same edges of $G$, that is, $W_V(G')=W(G)'$. Thus, we have
\begin{equation*}
\im(W(G))+k=\im(W(G)')=\reg(W(G)')=\reg(W(G))+k
\end{equation*}
so that $\im(W(G))=\reg(W(G))$. On the other hand, the equality $\im(W(G))=\alpha(G)$ can be easily verified. 
\end{proof}

%%%%%%%%%%%%%%%%%%%%%%%%%%%%%%%%%%%%%%%%%%%%%%%%%%%%%%%%%%%%%%%%%%%%%%%%%%%%%%%%%%%%%%%%%%%%%%%%%

\end{document}